\documentclass[reqno]{amsart}
\usepackage[english]{babel}
\usepackage{amssymb,graphicx}
\usepackage[T1,T5]{fontenc}

\newtheorem{theorem}{Theorem}[section]
\newtheorem{proposition}[theorem]{Proposition}
\newtheorem{lemma}[theorem]{Lemma}
\newtheorem{corollary}[theorem]{Corollary}

\theoremstyle{definition}
\newtheorem{definition}[theorem]{Definition}
\newtheorem{example}[theorem]{Example}

\theoremstyle{remark}
\newtheorem*{remark}{Remark}

\numberwithin{equation}{section}

\newcommand{\C}{\mbox{$\mathbb{C}$}}
\newcommand{\N}{\mbox{$\mathbb{N}$}}

\newcommand{\E}{\mbox{$\mathcal{E}$}}

\newcommand{\Ep}{\mbox{$\mathcal{E}_{p}$}}

\newcommand{\MM}{\mbox{$\mathcal{M}_{p,m}$}}

\newcommand{\EE}{\operatorname{E}_{w}}
\newcommand{\PP}{\operatorname{P}}
\newcommand{\be}{\beta^{n-m}}
\newcommand{\dd}{{\bf d}}

\newcommand{\HH}{\operatorname{H}_{m}}

\newcommand{\Jp}{\operatorname{J}_p}
\newcommand{\Ip}{\operatorname{I}_p}

\begin{document}

\author{Per \AA hag}
\address{Department of Mathematics and Mathematical Statistics\\ Ume\aa \ University\\ SE-901 87 Ume\aa \\ Sweden}
\email{Per.Ahag@math.umu.se}
\author{Rafa\l\ Czy{\.z}}
\address{Institute of Mathematics \\ Faculty of Mathematics and Computer Science \\ Jagiellonian University\\ \L ojasiewicza 6\\ 30-348 Krak\'ow\\ Poland}
\email{Rafal.Czyz@im.uj.edu.pl}\thanks{The second-named author was supported by the Priority Research Area SciMat under the program Excellence Initiative - Research University at the Jagiellonian University in Krak\'ow.}

\dedicatory{We raise our cups to Urban Cegrell, gone but not forgotten, gone but ever here. \\ Until we meet again in Valhalla!}


\title[Quasimetric spaces in generalized potential theory]{On a family of quasimetric spaces in generalized potential theory}

\begin{abstract}
We construct a family of quasimetric spaces in generalized potential theory containing $m$-subharmonic functions with finite $ (p,m)$-energy. These quasimetric spaces will be viewed both in $\C^n$ and in compact K\"ahler manifolds, and their convergence will be used to improve known stability results for the complex Hessian equations.
\end{abstract}

\maketitle

\begin{center}\bf
\today
\end{center}

\section{Introduction}

Although the abstract metric spaces introduced by Fr\'{e}chet at the beginning of the last century are of the utmost importance, in some applications, they are too restrictive and need a more general model. To name a few examples: the Minkowski $p$-distance in psychology ($p<1$)~\cite{JakelScholkopfWichmann,shepard}, the Zolotarev distance in spaces of random variables~\cite{Zolotarev} and the $d_{\epsilon}$-distance in machine learning~\cite{CalabuigFalcianiSancez}. The terminology has not yet stabilized within the many generalizations of metric spaces, and therefore let us determine which one we will work with here. Let $X$ be a non-empty set, and let $d:X\times X\to [0,\infty)$ be a function that satisfies:

\medskip

\begin{enumerate}\itemsep2mm
\item $d(x,y)=0$ if, and only if, $x=y$;

\item $d(x,y)=d(y,x)$, for all $x,y\in X$;

\item there exists a constant $C\geq 1$ such that
\[
d(x,y)\leq C(d(x,z)+d(z,y))
\]
for all $x,y,x\in X$.
\end{enumerate}
Following for example Heinonen~\cite{H}, we shall call $d$ for a \emph{quasimetric}, and the pair $(X,d)$ for a \emph{quasimetric space}. Some writers call this instead for a nearmetric
or inframetric. Next, let us define our specific $X$ and then construct $d$.

Let $n\geq 2$ and $1\leq m\leq n$. We say that a $\mathcal C^2$-function $u$ defined in a bounded domain in $\mathbb C^n$ is \emph{$m$-subharmonic} if the elementary symmetric functions are positive $\sigma_l(\lambda(u))\geq 0$  for $l=1,\dots,m$, where $\lambda(u)=(\lambda_1,\ldots,\lambda_n)$ are eigenvalues of the complex Hessian matrix $D_{\mathbb{C}}^2u=[\frac {\partial ^2u}{\partial z_j\partial \bar z_k}]$. The complex $m$-Hessian operator on a $\mathcal C^2$-function $u$ is then defined by
\[
\operatorname H_m(u)=c(n,m)\sigma_m(\lambda(D_{\mathbb{C}}^2u)),
\]
for some constant $c(n,m)$ depending only on $n$ and $m$.

This construction yields that the $1$-Hessian operator is the Laplace operator defined on $1$-subharmonic functions that
are just the subharmonic functions, while the complex $n$-Hessian operator is the complex Monge-Amp\`{e}re operator defined on $n$-subharmonic functions that are the plurisubharmonic functions. Historically this model goes back to Caffarelli, Nirenberg, and Spruck~\cite{CNS85} in 1985, where they did a similar construction for the real Hessian matrix. Vinacua, a student of Nirenberg, was one of those who adapted the idea of the Hessian operator to the complex setting (\cite{vinacua1,vinacua2}) that we shall use here. Later in 2005, B\l ocki~\cite{Blocki} introduced pluripotential methods to the theory of complex Hessian operators, and there he, among other things,  generalized the complex Hessian operator to non-smooth $m$-subharmonic functions. For $p>0$, set
\[
e_{p,m}(u)=\int_{\Omega} (-u)^p \operatorname H_m(u),
\]
and we shall call $e_{p,m}(u)$ for the \emph{$(p,m)$-energy of the function $u$}. Let $\mathcal E_{p,m}(\Omega)$ be the class of $m$-subharmonic functions that, in a general sense, vanishes on the boundary and additionally they should have finite $(p,m)$-energy. The classes $\mathcal E_{p,m}(\Omega)$ are sometime known as the Cegrell's generalized energy classes, after Cegrell's influential work~\cite{cegrell_pc} on $\mathcal E_{p,n}(\Omega)$. For the early work on the theory of variation for the complex $n$-Hessian operator see e.g.~\cite{BedfordTaylor78,BedfordTaylor79,ChernLevineNirenberg,Gaveau1,Gaveau2,Kalina}. On the other hand, if $m=1$, and $p=1$, then $e_{1,1}$ is the Dirichlet energy integral from potential theory connected to a long and fruitful history.

Set $X=\mathcal E_{p,m}(\Omega)$, and let $\Jp:X\times X\to [0,\infty)$ be defined by
\[
\Jp(u,v)=\left(\int|u-v|^p(\HH(u)+\HH(v))\right)^{\frac 1{p+m}}.
\]
In Theorem~\ref{thm_quasimetric}, we prove that $(X,\Jp)$ is a quasimetric space in the above sense, and in Theorem~\ref{comp} we
prove that it is complete. Later in Section~\ref{sec_kah} we shall consider the compact K\"ahler manifold case, and in Theorem~\ref{thm_quasimetric}, and Theorem~\ref{thmKah_comp}, we shall prove that the corresponding construction is a complete quasimetric space. Guedj, Lu, and Zeriahi~\cite[Theorem~1.6]{GLZ} proved the quasi-triangle inequality in the case $m=n$, in the compact K\"ahler manifold setting (see also~\cite[Theorem~1.8]{BBEGZ19}, and~\cite{NessaLu}).

\medskip

In Section 3, we will use the complete quasimetric space $(X,\Jp)$ in $\mathbb C^n$ to prove the following stability results for the complex
Hessian operators. First, let
\[
\begin{aligned}
\MM=\big\{&\mu \; :\;  \mu \text { is a non-negative Radon measure on $\Omega$ such that }\\
 & \HH(u)=\mu \text{ for some } u\in \mathcal E_{p,m}(\Omega)\big \},
\end{aligned}
\]
and then recall the following characterization of $\MM$:

\medskip

\begin{enumerate}\itemsep2mm
\item $\mu\in \MM$;
\item there exists a constant $A\geq 0$ such that
\[
\int_{\Omega}(-u)^p\,d\mu\leq A\, e_{p,m}(u)^{\frac{p}{p+m}}\quad \text{ for all } u\in\mathcal E_{p,m}(\Omega)\, ;
\]
\item $\mathcal E_{p,m}(\Omega)\subset L^p(\mu)$;

\item there exists unique function $U(\mu)\in \mathcal E_{p,m}(\Omega)$ the solution to the Dirichlet problem for the complex Hessian equation $\HH(U(\mu))=\mu$.

\end{enumerate}
(see~\cite{cegrell_pc,L3} for details). Let $\mu\in \MM$, then in Theorem~\ref{thm_stability}, we prove that if
$0\leq f, f_j\leq 1$ are measurable functions such that $f_j\to f$ in $L^1_{loc}(\mu)$, as $j\to \infty$, then $\Jp(U(f_j\mu), U(f\mu))\to 0$, $j\to \infty$. By Proposition~\ref{cap}
we know that convergence in $(X,\Jp)$ implies convergence in capacity, but by Example~\ref{conv_example} we have that the converse statement is false. Hence, Theorem~\ref{thm_stability}
is a generalization of~\cite[Theorem~7.2]{thien2}. Note that this also implies improved results in the pluricomplex case, $m=n$, and therefore Theorem~\ref{thm_stability} also generalizes the stability result by Cegrell and Ko\l odziej~\cite{CK}. For further information about these types of stability results in the case $m=n$ we refer to~\cite[Section~7.2]{C}.

\bigskip

\section{Preliminaries}\label{sec_prelim}

Here, we shall present some crucial and necessary facts about $m$-subharmonic functions that shall be used in this manuscript. For further information, see e.g.~\cite{SA,AS2,L}. First, let $n\geq 2$, $1\leq m\leq n$, and let $\Omega$ be a bounded domain in $\C^n$.  Then define $\mathbb C_{(1,1)}$ to be the set of $(1,1)$-forms with constant coefficients, and set
\[
\Gamma_m=\left\{\alpha\in \mathbb C_{(1,1)}: \alpha\wedge \beta^{n-1}\geq 0, \dots , \alpha^m\wedge \beta ^{n-m}\geq 0   \right\}\, ,
\]
where $\beta=dd^c|z|^2$ is the canonical K\"{a}hler form in $\C^n$. We then say that a  subharmonic function $u$ defined on $\Omega$ is \emph{$m$-subharmonic}, if the following inequality holds
\[
dd^cu\wedge\alpha_1\wedge\dots\wedge\alpha_{m-1}\wedge\beta^{n-m}\geq 0\, ,
\]
in the sense of currents for all $\alpha_1,\ldots,\alpha_{m-1}\in \Gamma_m$. Furthermore, we call $\Omega$ for \emph{$m$-hyperconvex} if it admits an exhaustion function $\varphi$  that is negative and $m$-subharmonic, i.e. the closure of the  set $\{z\in\Omega : \varphi(z)<c\}$ is compact in $\Omega$, for every $c\in (-\infty, 0)$. For further information about $m$-hyperconvex domains we refer to~\cite{ACH}.

Let $p>0$. We say that a $m$-subharmonic function $\varphi$ defined on a $m$-hyperconvex domains $\Omega$ belongs to:

\medskip

\begin{itemize}\itemsep2mm
\item[$(i)$] $\E_{0,m}(\Omega)$ if, $\varphi$ is bounded,
\[
\lim_{z\rightarrow\xi} \varphi (z)=0 \quad \text{ for every } \xi\in\partial\Omega\, ,
\]
and
\[
\int_{\Omega} \HH(\varphi)<\infty\, ,
\]
where $\HH(u)=(dd^cu)^m\wedge \be$ is the complex Hessian operator.
\item[$(ii)$]  $\E_{p,m}(\Omega)$ if, there exists a decreasing sequence, $\{u_{j}\}$, $u_{j}\in\E_{0,m}(\Omega)$,
that converges pointwise to $u$ on $\Omega$, as $j$ tends to $\infty$, and
\[
\sup_{j} e_{p,m}(u_j)=\sup_{j}\int_{\Omega}(-u_{j})^p\HH(u_j)< \infty\, .
\]
\end{itemize}

\begin{theorem}\label{thm_holder} Let $n\geq 2$, $1\leq m\leq n$, and assume that $\Omega\subset \mathbb C^n$ is a $m$-hyperconvex domain. There exists a constant $D(m,p)$ (depending only on $p$ and $m$) such that for any $u_0,u_1,\ldots ,
u_m\in\E_{p,m}(\Omega)$ it holds
\begin{multline*}
\int_\Omega (-u_0)^p dd^c u_1\wedge\cdots\wedge dd^c u_m\wedge \be\\ \leq
\; D(m,p) e_{p,m}(u_0)^{\frac {p}{m+p}}e_{p,m}(u_1)^{\frac {1}{m+p}}\cdots
e_{p,m}(u_m)^{\frac {1}{m+p}}\, .
\end{multline*}
\end{theorem}
\begin{proof}
See Theorem~3.4 in~\cite{persson} (see also~\cite{czyz_energy,cegrell_pc,CP}).
\end{proof}

\begin{theorem}\label{thm_prel} Let $n\geq 2$, $1\leq m\leq n$, and assume that $\Omega\subset \mathbb C^n$ is a $m$-hyperconvex domain.
Furthermore, assume that $u,v\in \mathcal E_{p,m}(\Omega)$, and $T$ be a positive closed current. Then it holds:
\begin{enumerate}\itemsep2mm

\item
\[
\int_{\{u<v\}}\HH(v)\leq \int_{\{u<v\}}\HH(u).
\]
\item If $\HH(v)\leq \HH(u)$, then $u\leq v$.

\item  If $\HH(u)(u<v)=0$, then $u\geq v$.

\item $\chi_{\{u<v\}}(dd^c\max(u,v))\wedge T=\chi_{\{u<v\}}(dd^cv)\wedge T$.

\item $\HH(\max(u,v))\geq \chi_{\{u\geq v\}}\HH(u)+ \chi_{\{u<v\}}\HH(v)$.
\end{enumerate}
\end{theorem}
\begin{proof} See~\cite{L3}.
\end{proof}

We shall need a comparison principle with weights. Proposition~\ref{cp with weights} will be used in the proof of Proposition~\ref{basic}.

\begin{proposition}\label{cp with weights} Let $n\geq 2$, $1\leq m\leq n$, and assume that $\Omega\subset \mathbb C^n$ is a $m$-hyperconvex domain. Assume that $u,v,w\in \mathcal E_{p,m}(\Omega)$ are such that $w\geq u\geq v$, then
\[
\int_{\Omega}(w-u)^p\HH(u)\leq (\max(p,1)+1)^m\int_{\Omega}(w-v)^p\HH(v).
\]
\end{proposition}
\begin{proof} Let  $u_1=u-w$, $v_1=v-w$ and $T=(dd^cw+dd^cu_1)^{m-1}\wedge \be$. Then we have
\begin{multline*}
\int_{\Omega}(-v_1)^p(dd^cw+dd^cu_1)\wedge T=\int_{\Omega}(-v_1)^pdd^cw\wedge T+\\
\int_{\Omega}(-v_1)^pdd^cu_1\wedge T=I_1+I_2.
\end{multline*}
Note that for $p\geq 1$
\begin{multline*}
dd^c(-(-v_1)^p)=p(1-p)(-v_1)^{p-2}dv_1\wedge d^cv_1+p(-v_1)^{p-1}dd^cv_1\leq \\
p(-v_1)^{p-1}(dd^cv_1+dd^c w),
\end{multline*}
and for $p<1$
\begin{multline*}
dd^c(-(-v_1)^p)=p(1-p)(-v_1)^{p-2}dv_1\wedge d^cv_1+p(-v_1)^{p-1}dd^cv_1\leq \\
p(1-p)(-v_1)^{p-2}dv_1\wedge d^cv_1+p(-v_1)^{p-1}(dd^cv_1+dd^c w),
\end{multline*}
Then we get
\begin{multline*}
I_1=\int_{\Omega}(-v_1)^pdd^cw\wedge T\leq \int_{\Omega}(-v_1)^pdd^cw\wedge T+\\
p\int_{\Omega}(-v_1)^{p-1}dv_1\wedge d^cv_1\wedge T=\int_{\Omega}(-v_1)^p(dd^cw+dd^cv_1)\wedge T.
\end{multline*}
For $p\geq 1$
\begin{multline*}
I_2=\int_{\Omega}(-v_1)^pdd^cu_1\wedge T= \int_{\Omega}(-u_1)dd^c(-(-v_1)^p)\wedge T\leq\\
p\int_{\Omega}(-u_1)(-v_1)^{p-1}(dd^cv_1+dd^c w)\wedge T\leq p \int_{\Omega}(-v_1)^{p}(dd^cv_1+dd^c w)\wedge T,
\end{multline*}
and for $p<1$
\begin{multline*}
I_2=\int_{\Omega}(-v_1)^pdd^cu_1\wedge T= \int_{\Omega}(-u_1)dd^c(-(-v_1)^p)\wedge T\leq\\
\int_{\Omega}(-u_1)\left(p(1-p)(-v_1)^{p-2}dv_1\wedge d^cv_1+p(-v_1)^{p-1}(dd^cv_1+dd^c w)\right)\wedge T\leq \\
\int_{\Omega}p(1-p)(-v_1)^{p-1}dv_1\wedge d^cv_1\wedge T+p(-v_1)^{p}(dd^cv_1+dd^c w)\wedge T\leq\\
(1-p)\int_{\Omega}(-v_1)^pdd^cv_1\wedge T+p\int_{\Omega}(-v_1)^p(dd^cv_1+dd^c w)\wedge T\leq \\
\int_{\Omega}(-v_1)^{p}(dd^cv_1+dd^c w)\wedge T.
\end{multline*}

Finally, for any $p>0$
\begin{multline*}
\int_{\Omega}(-u_1)^p(dd^cw+dd^cu_1)\wedge T\leq \int_{\Omega}(-v_1)^p(dd^cw+dd^cu_1)\wedge T\leq \\
(\max(p,1)+1)\int_{\Omega}(-v_1)^p(dd^cw+dd^cv_1)\wedge T\leq \dots\leq \\ (\max(p,1)+1)^m\int_{\Omega}(w-v)^p(dd^cv)^m\wedge \be.
\end{multline*}
\end{proof}

We end this section with a Xing type inequality that shall be used in Proposition~\ref{basic}. In Proposition~\ref{cp xing_m}, we shall as well prove the correspondent result in the case of compact K\"ahler manifolds.

\begin{proposition}\label{cp xing} Let $n\geq 2$, $1\leq m\leq n$, and assume that $\Omega\subset \mathbb C^n$ is a $m$-hyperconvex domain, and $u,v\in \mathcal E_{p,m}(\Omega)$.
\begin{enumerate}
\item If $u\leq v$, then
\[
\int_{\Omega}(v-u)^p\HH(v)\leq \int_{\Omega}(v-u)^p\HH(u).
\]
\item Without any additional assumption on $u$, and $v$,  it holds
\[
\int_{\{u<v\}}(v-u)^p\HH(v)\leq \int_{\{u<v\}}(v-u)^p\HH(u).
\]
\end{enumerate}
\end{proposition}
\begin{proof}
(1) Let $\epsilon >1$, then $\epsilon u<u\leq v$. We obtain
\begin{multline*}
\int_{\Omega}(v-\epsilon u)^p(\HH(\epsilon u)-\HH(v))=\\
\sum_{k+l=m-1}\int_{\Omega}(v-\epsilon u)^pdd^c(\epsilon u-v)\wedge(dd^c\epsilon u)^k\wedge(dd^cv)^l\wedge \be=\\
p\sum_{k+l=m-1}\int_{\Omega}(v-\epsilon u)^{p-1}d(v-\epsilon u)\wedge d^c(v-\epsilon u)\wedge(dd^c\epsilon u)^k\wedge(dd^cv)^l\wedge \be \\ \geq 0.
\end{multline*}
From this it follows
\[
\int_{\Omega}(v-\epsilon u)^p\HH(v)\leq \epsilon ^m\int_{\Omega}(v-\epsilon u)^p\HH(u),
\]
and then by using the monotone convergence theorem, and finally passing to the limit, $\epsilon \to 1^+$, we arrive at the desired conclusion.

\medskip

\noindent  (2) From (1), and Theorem~\ref{thm_prel}, we obtain
\begin{multline*}
\int_{\{u<v\}}(v-u)^p\HH(v)=\int_{\{u<v\}}(\max(u,v)-u)^p\HH(\max(u,v))=\\
\int_{\Omega}(\max(u,v)-u)^p\HH(\max(u,v))\leq \int_{\Omega}(\max(u,v)-u)^p\HH(u)=\\
\int_{\{u<v\}}(v-u)^p\HH(u).
\end{multline*}
\end{proof}

\section{Quasimetric spaces}\label{sec_quasimetric}

Let $X$ be a non-empty set, and let $d:X\times X\to [0,\infty)$ be a function that satisfies:

\begin{enumerate}\itemsep2mm
\item $d(x,y)=0$ if, and only if, $x=y$;

\item $d(x,y)=d(y,x)$, for all $x,y\in X$;

\item there exists a constant $C\geq 1$ such that
\[
d(x,y)\leq C(d(x,z)+d(z,y))
\]
for all $x,y,x\in X$.
\end{enumerate}

\noindent In this paper we shall call $d$ for a \emph{quasimetric}, and the pair $(X,d)$ for a \emph{quasimetric space}. Recall that every metric is a quasimetric. Furthermore, in every quasimetric space $(X,d)$ there exists a metric $\rho$ with the property that there is an $\epsilon>0$, and a constant $A>0$ such that
\[
A^{-1}d^{\epsilon}\leq \rho\leq Ad^{\epsilon}.
\]

In the next definition we shall define a functional, $\Jp$, in $\mathcal E_{p,m}(\Omega)\times\mathcal E_{p,m}(\Omega)$. After proving some elementary properties of $\Jp$ in Proporition~\ref{basic}, and that $\Jp$ satisfies the quasi-triangle inequality (Lemma~\ref{qm}), we can in Theorem~\ref{thm_quasimetric} conclude that
we have a family of quasimetric spaces. These spaces are complete as shall be shown in Theorem~\ref{comp}.

\begin{definition}\label{def_Jp} Let $n\geq 2$, $1\leq m\leq n$, and assume that $\Omega\subset \mathbb C^n$ is a $m$-hyperconvex domain. For $u,v\in \mathcal E_{p,m}(\Omega)$ and $p>0$ let us define
\[
\Jp(u,v)=\left(\int_{\Omega}|u-v|^p(\HH(u)+\HH(v))\right)^{\frac 1{p+m}}.
\]
\end{definition}

In the next definition, let us recall the rooftop envelope.

\begin{definition} Let $n\geq 2$, $1\leq m\leq n$, and assume that $\Omega\subset \mathbb C^n$ is a $m$-hyperconvex domain. For $u_1,\dots,u_k\in \mathcal E_{p,m}(\Omega)$ define
\[
\PP(u_1,\dots,u_k)=\Big(\sup\{\varphi\in \mathcal E_{p,m}(\Omega): \varphi\leq \min(u_1,\dots,u_k)\}\Big)^*,
\]
where  $(\,)^*$ is the upper semicontinuous regularization.
\end{definition}
\begin{remark}
If $u,v\in \mathcal E_{p,m}(\Omega)$, then $u+v\leq \PP(u,v)$, and therefore we have  $P(u,v)\in \mathcal E_{p,m}(\Omega)$.
\end{remark}

We shall need the following minimum principle. In~\cite[Theorem~4.3]{ACmetric} Theorem~\ref{mp} was proved for the class $\mathcal E_{1,m}(\Omega)$,
but the proof without any change goes over to $\mathcal E_{p,m}(\Omega)$. Therefore we omit the proof here.

\begin{theorem}\label{mp} Let $n\geq 2$, $1\leq m\leq n$, and assume that $\Omega\subset \mathbb C^n$ is a $m$-hyperconvex domain. Let $u,v\in \mathcal E_{p,m}(\Omega)$.
Then the following holds
\begin{equation}\label{min}
\HH(\PP(u,v))\leq \chi_{\{\PP(u,v)=u\}}\HH(u)+\chi_{\{\PP(u,v)=v\}}\HH(v).
\end{equation}
\end{theorem}

\begin{proposition}\label{basic} Let $n\geq 2$, $1\leq m\leq n$, and assume that $\Omega\subset \mathbb C^n$ is a $m$-hyperconvex domain. Furthermore assume that  $u,v,w\in \mathcal E_{p,m}(\Omega)$.
\begin{enumerate}\itemsep2mm
\item $\Jp(u,v)<\infty$;

\item $\Jp(u,v)=0$ if, and only if, $u=v$;

\item $\Jp(u,v)=\Jp(v,u)$;

\item $\Jp(u,v)^{p+m}=\Jp(u,\max (u,v))^{p+m}+\Jp(v,\max(u,v))^{p+m}$;

\item \begin{multline*}
\max(\Jp(u,\max (u,v)),\Jp(v,\max(u,v)))\leq \Jp(u,v)\leq \\ \Jp(u,\max (u,v))+\Jp(v,\max(u,v));
\end{multline*}

\item If $u\leq v$, then
\[
2\int_{\Omega}(v-u)^p\HH(v)\leq \Jp(u,v)^{p+m}\leq 2\int_{\Omega}(v-u)^p\HH(u);
\]
\item If $u\leq v\leq w$, then $\Jp(u,v)\leq 2^{\frac {p+2}{p+m}}\Jp(u,w)$;

\item If $u\leq v\leq w$, then $\Jp(v,w)^{p+m}\leq (\max(p,1)+1)^m\Jp(u,w)^{p+m}$;

\item $\Jp(v,\PP(u,v))\leq \Jp(u,\max (u,v))\leq \Jp(u,v)$;

\item $\Jp(v,\PP(u,v))^{p+m}+\Jp(u,\PP(u,v))^{p+m}\leq \Jp(v,u)^{p+m}$;

\item If $u\leq v$, then $\Jp(\PP(u,w),\PP(v,w))\leq 2^{\frac {p+2}{p+m}}\Jp(u,v)$.

\end{enumerate}
\end{proposition}
\begin{proof}

\smallskip

\noindent (1). By Theorem~\ref{thm_holder} we have
\begin{multline*}
\Jp(u,v)^{p+m}\leq \int_{\Omega}(-u-v)^p (\HH(u)+\HH(v))\leq \\ D(m,p)e_{p,m}(u+v)^{\frac {p}{p+m}}(e_{p}(u)^{\frac {m}{p+m}}+e_{p}(v)^{\frac {m}{p+m}})<\infty.
\end{multline*}

\medskip

\noindent (2). It is obvious that $\Jp(u,u)=0$. Next, assume that $\Jp(u,v)=0$. Then $\HH(u)(\{u<v\})=0$, so by Theorem~\ref{thm_prel} we obtain $u\geq v$. In a similar manner, we have $\HH(v)(\{v<u\})=0$. Hence, $v\leq u$, and therefore it follows $u=v$.

\medskip

\noindent (3). This property is an immediate consequence of the definition of $\Jp$.

\medskip

\noindent (4). Thanks to Theorem~\ref{thm_prel} it follows that $\HH(\max(u,v))=\HH(u)$ on the set $\{u>v\}$, and similarly $\HH(\max(u,v)))=\HH(v)$ on the set $\{v>u\}$. Thus,
\begin{multline*}
\Jp(u,v)^{p+m}=\int_{\Omega}|u-v|^p(\HH(u)+\HH(v))=\\
\int_{\{u<v\}}(\max(u,v)-u)^p(\HH(u)+\HH(\max(u,v)))+\\
\int_{\{v<u\}}(\max(u,v)-v)^p(\HH(v)+\HH(\max(u,v)))=\\
\Jp(u,\max(u,v))^{p+m}+\Jp(v,\max(u,v))^{p+m}.
\end{multline*}

\medskip

\noindent (5). This is an immediate consequence of (4).

\medskip

\noindent (6). Proposition~\ref{cp xing} yields this result.

\medskip

\noindent (7). Note that $0\leq w-v\leq w-u$. Then by using (6) we get
\begin{multline*}
\Jp(u,v)^{p+m}\leq 2\left(\int_{\Omega}(v-u)^p\HH(u)\right)\leq\\
2^{1+p}\left(\int_{\Omega}(w-v)^p\HH(u)+\int_{\Omega}(w-u)^p\HH(u)\right)\leq\\
2^{2+p}\Jp(u,w)^{p+m}.
\end{multline*}

\medskip

\noindent (8). By Proposition~\ref{cp with weights} we get
\begin{multline*}
\Jp(v,w)^{p+m}=\int_{\Omega}(w-v)^p(\HH(v)+\HH(w))=\\
\int_{\Omega}(w-v)^p\HH(v)+\int_{\Omega}(w-v)^p\HH(w)\leq\\
(\max(p,1)+1)^m\int_{\Omega}(w-u)^p\HH(u)+\int_{\Omega}(w-u)^p\HH(w)=\\ (\max(p,1)+1)^m\Jp(u,w)^{p+m}.
\end{multline*}

\medskip

\noindent (9). By Theorem~\ref{thm_prel} we get
\begin{multline*}
\Jp(u,\max(u,v))^{p+m}=\int_{\Omega}(\max(u,v)-u)^p(\HH(u)+\HH(\max(u,v)))=\\
\int_{\{u<v\}}(v-u)^p(\HH(u)+\HH(\max(u,v))))=\int_{\{u<v\}}(v-u)^p(\HH(u)+\HH(v)).
\end{multline*}
On the other hand, using Theorem~\ref{mp} we arrive at
\begin{multline*}
\Jp(v,\PP(u,v))^{p+m}=\int_{\Omega}(v-\PP(u,v))^p(\HH(v)+\HH(\PP(u,v)))\leq\\
\int_{\{\PP(u,v)<v\}}(v-\PP(u,v))^p\Big(\HH(v)+ \chi_{\{\PP(u,v)=v\}}\HH(v)+\chi_{\{\PP(u,v)=u\}}\HH(u))\Big)=\\
\int_{\{\PP(u,v)<v\}\cap \{\PP(u,v)=u\}}(v-\PP(u,v))^p(\HH(u)+\HH(v))\leq\\
\int_{\{u<v\}}(v-u)^p(\HH(u)+\HH(v))=\Jp(u,\max (u,v))^{p+m}.
\end{multline*}
The last inequality follows from (5).

\medskip

\noindent (10). This follows from (4) together with (9).

\medskip

\noindent (11). Note that $u\leq \max(u,\PP(v,w))\leq v$, and then by (7) and (9), we have
\begin{multline*}
\Jp(\PP(v,w),\PP(u,w))=\Jp(\PP(v,w),\PP(u,\PP(v,w)))\leq \Jp(u,\max(u,\PP(v,w)))\leq\\
2^{\frac {p+2}{p+m}}\Jp(u,v).
\end{multline*}
\end{proof}

By letting us be inspired by~\cite{GLZ}, we can in the next lemma prove that $\Jp$ enjoys the quasi-triangle inequality.

\begin{lemma}\label{qm} Let $n\geq 2$, $1\leq m\leq n$, and assume that $\Omega\subset \mathbb C^n$ is a $m$-hyperconvex domain. Then there exists $C>0$ such that for any $u,v,w\in \mathcal E_{p,m}(\Omega)$ it holds
\[
\Jp(u,v)\leq C(\Jp(u,w)+\Jp(w,v)).
\]
Furthermore, the constant $C$ can be taken as $C=\left(2^{2p+1}(2^p+1)3^{m}\right)^{\frac 1{p+m}}$.
\end{lemma}

\begin{proof}
By using the comparison principle (see e.g. Theorem~\ref{thm_prel}) it follows
\[
\begin{aligned}
\HH(u)(\{v<u-2s\})&\leq \HH(v)(\{v<u-2s\})\\
\HH(v)(\{u<v-2s\})&\leq \HH(u)(\{u<v-2s\}),
\end{aligned}
\]
and therefore it holds
\begin{multline}\label{2}
\Jp(u,v)^{p+m}=\int_{\Omega}|u-v|^p(\HH(u)+\HH(v))=\\
p\int_0^{\infty}s^{p-1}(\HH(u)+\HH(v))(\{|u-v|>s\})ds=\\
p2^p\int_0^{\infty}s^{p-1}(\HH(u)+\HH(v))(\{|u-v|>2s\})ds\leq \\
p2^{p+1}\int_0^{\infty}s^{p-1}(\HH(u)(\{u<v-2s\})+\HH(v)(\{v<u-2s\}))ds.
\end{multline}

Next we shall estimate the measure $\HH(u)(\{u<v-2s\})$. Since
\[
\{u<v-2s\}\subset \{u<w-s\}\cup \{w-s\leq u<v-2s\}\subset \{u<w-s\}\cup \left\{w<\frac {u+2v}{3}-\frac s3\right\},
\]
we can using again the comparison principle (see e.g. Theorem~\ref{thm_prel}) and arrive at
\begin{multline*}
\HH(u)(\{w-s\leq u<v-2s\})\leq \HH(u)\left(\left\{w<\frac {u+2v}{3}-\frac s3\right\}\right)\leq \\
3^m\HH\left(\frac {u+2v}{3}\right)\left(\left\{w<\frac {u+2v}{3}-\frac s3\right\}\right)\leq 3^m\HH(w)\left(\left\{w<\frac {u+2v}{3}-\frac s3\right\}\right).
\end{multline*}
Note that also holds
\begin{multline*}
\left|w-\frac {u+2v}{3}\right|^p=\frac {1}{3^p}|3w-u-2v|^p\leq\\
\left(\frac {2}{3}\right)^p(|w-u|^p+|2w-2v|^p)\leq \left(\frac 23\right)^p|w-u|^p+\left(\frac 43\right)^p|w-v|^p.
\end{multline*}
Finally, by using the above estimates
\begin{multline*}
p\int_0^{\infty}s^{p-1}\HH(u)(\{u<v-2s\})ds\leq p\int_0^{\infty}s^{p-1}\HH(u)(\{u<w-s\})ds+\\
p3^m\int_0^{\infty}s^{p-1}\HH(w)\left(\left\{w<\frac {u+2v}{3}-\frac s3\right\}\right)ds\leq\\
\int_{\Omega}|w-u|^p\HH(u)+3^{m+p}\int_{\Omega}\left|w-\frac {u+2v}{3}\right|^p\HH(w)\leq\\
\int_{\Omega}|w-u|^p\HH(u)+3^{m}2^p\int_{\Omega}|w-u|^p\HH(w)+3^m4^p\int_{\Omega}|w-v|^p\HH(w).
\end{multline*}
A similar estimate can be obtain for $\HH(v)(\{v<u-2s\})$, and therefore by (\ref{2}) we get
\begin{multline*}
\Jp(u,v)^{p+m}\leq p2^{p+1}\int_0^{\infty}s^{p-1}(\HH(u)(\{u<v-2s\})+\HH(v)(\{v<u-2s\}))ds\leq\\
2^{p+1}\left(\int_{\Omega}|w-u|^p\HH(u)+3^{m}2^p\int_{\Omega}|w-u|^p\HH(w)+3^m4^p\int_{\Omega}|w-v|^p\HH(w)\right)+\\
2^{p+1}\left(\int_{\Omega}|w-v|^p\HH(v)+3^{m}2^p\int_{\Omega}|w-v|^p\HH(w)+3^m4^p\int_{\Omega}|w-u|^p\HH(w)\right)\leq \\
2^{2p+1}(2^p+1)3^{m}(\Jp(u,w)^{p+m}+\Jp(v,w)^{p+m}).
\end{multline*}
To finish the proof it enough to observe that
\[
\Jp(u,v)\leq \left(2^{2p+1}(2^p+1)3^{m})\right)^{\frac 1{p+m}}(\Jp(u,w)+\Jp(v,w)).
\]
\end{proof}

Thanks to Proposition~\ref{basic} and Lemma~\ref{qm} we can now conclude that we have a family of quasimetric spaces. The aim of the rest of this
section is to prove that they are complete, which we will do in Theorem~\ref{comp}.

\begin{theorem}\label{thm_quasimetric} Let $n\geq 2$, $1\leq m\leq n$, and assume that $\Omega\subset \mathbb C^n$ is a $m$-hyperconvex domain. The pair $(\mathcal E_{p,m}(\Omega),\Jp)$ is a quasimetric space.
\end{theorem}

To be able to prove that the quasimetric spaces are complete we need information on how $\Jp$ behaves under monotone sequences. In the case $p=1$, $\operatorname{J}_1$ is continuous both for increasing and decreasing sequences, but only for decreasing sequences when $p\neq 1$.

\begin{proposition}\label{monotone} Let $n\geq 2$, $1\leq m\leq n$, and assume that $\Omega\subset \mathbb C^n$ is a $m$-hyperconvex domain. Let $u_j\in \mathcal E_{p,m}(\Omega)$ be a decreasing sequence that converging to $u \in\mathcal E_{p,m}(\Omega)$. Then $\Jp(u_j,u)\to 0$, as $j\to\infty$. If $p=1$, then the same statement is true for increasing sequences.
\end{proposition}
\begin{proof}
First assume that the sequence $u_j$ is decreasing. Then by Proposition~\ref{basic}, and the monotone convergence theorem we get
\[
\Jp(u_j,u)^{p+m}\leq 2\int_{\Omega}(u_j-u)^p\HH(u)\to 0,
\]
as $j\to \infty$. Now assume that $p=1$, and  the sequence $u_j$ is increasing. From~\cite{ACmetric} and Proposition~\ref{basic} we get
\[
\Jp(u_j,u)^{p+m}\leq 2\int_{\Omega}(u-u_j)\HH(u_j)\to 0,
\]
as $j\to \infty$.
\end{proof}

To proof that the space $(\mathcal E_{p,m}(\Omega),\Jp)$ is complete we shall need the following elementary fact.

\begin{proposition}\label{weakconvergence} If $f_j$ is an increasing sequence of continuous functions defined on $X$ such that $f_j\nearrow f$, and  $\mu_j\to \mu$ weakly, as $j\to \infty$, then
\[
\liminf_{j\to \infty}\int_{X}f_jd\mu_j\geq \int_{X}fd\mu.
\]
\end{proposition}
\begin{proof}
First we shall prove that if $\alpha$ is a lower continuous function then
\begin{equation}\label{9}
\int_{X}\alpha d\mu\leq \liminf_{j\to \infty}\int_{X}\alpha d\mu_j.
\end{equation}
Let $\mathcal C_0(X)\ni g_k\nearrow \alpha$, then $\int_{X}g_kd\mu_j\leq \int_{X}\alpha d\mu_j$. Now by the weak convergence we get
\[
\int_{X}g_kd\mu=\lim_{j\to \infty}\int_{X}g_kd\mu_j\leq \liminf_{j\to \infty}\int_{X}\alpha d\mu_j\,
\]
and by monotone convergence theorem we obtain (\ref{9}).

Fix $k$, and let $j\geq k$, then $\int_{X}f_jd\mu_j\geq \int_{X}f_kd\mu_j$. Therefore by (\ref{9})
\[
\liminf_{j\to \infty}\int_{X}f_jd\mu_j\geq \liminf_{j\to \infty}\int_{X}f_kd\mu_j\geq \int_{X}f_kd\mu.
\]
From the monotone convergence theorem it now follows
\[
\liminf_{j\to \infty}\int_{X}f_jd\mu_j\geq \lim_{k\to \infty}\int_{X}f_kd\mu=\int_{X}fd\mu.
\]
\end{proof}

The completeness of $(\mathcal E_{p,m}(\Omega),\Jp)$ is next on our agenda.

\begin{theorem}\label{comp} Let $n\geq 2$, $1\leq m\leq n$, and assume that $\Omega\subset \mathbb C^n$ is a $m$-hyperconvex domain. The quasimetric space $(\mathcal E_{p,m}(\Omega),\Jp)$ is complete.
\end{theorem}
\begin{proof}
Let $\{\varphi_j\}\subset \mathcal E_{p,m}(\Omega)$ be a Cauchy sequence. After choosing a subsequence we may assume that $\Jp(\varphi_j,\varphi_{j+1})\leq \frac {1}{3(2C)^{j+3}}$ for $j\in \mathbb N$, where $C$ is the constant from quasi-triangle inequality. From~\cite[Theorem~5.2]{ACH} it follows that for each $\varphi_j$ there exists decreasing sequence of continuous functions $\mathcal E_{0,m}(\Omega)\cap \mathcal C(\overline \Omega)\ni\psi_j^k\searrow \varphi_j$, $k\to \infty$. Then we can choose $u_j=\psi_j^{k(j)}$ such that
\begin{equation}\label{8}
\Jp(u_j,\varphi_j)\leq  \frac {1}{3(2C)^{j+3}},
\end{equation}
see Proposition~\ref{monotone}. Observe that $\{u_j\}\subset \mathcal E_{p,m}(\Omega)$ is also a Cauchy sequence. For each $j\in \mathbb N$ we have
\begin{multline}\label{3}
\Jp(u_j,u_{j+1})\leq C^2(\Jp(u_j,\varphi_j)+\Jp(\varphi_j,\varphi_{j+1})+\Jp(\varphi_{j+1},u_{j+1}))\leq \\ \frac {3C^2}{3(2C)^{j+3}}\leq \frac {1}{(2C)^{j+1}}.
\end{multline}
From the quasi-triangle inequality together with (\ref{3}) we get
\begin{multline}\label{4}
\Jp(0,u_j)\leq C\Jp(0,u_1)+C^2\Jp(u_1,u_2)+\dots+C^j\Jp(u_{j-1},u_j)\leq \\
C\Jp(0,u_1)+\frac {C^2}{(2C)^2}+\dots+\frac {C^j}{(2C)^{j}}\leq C\Jp(0,u_1)+1.
\end{multline}
Set $v_{j,k}=\max(u_j,\dots,u_{k})$, for $k\geq j$. Since $v_{j,k}\in \mathcal E_{p,m}(\Omega)$, we can use Proposition~\ref{basic} and (\ref{3}) to arrive at
\begin{multline}\label{5}
\Jp(u_j,v_{j,k})=\Jp(u_j,\max(u_j,v_{j+1,k}))\leq \Jp(u_j,v_{j+1,k})\leq \\
C(\Jp(u_j,u_{j+1})+\Jp(u_{j+1},v_{j+1,k}))= \\ C(\Jp(u_j,u_{j+1})+\Jp(u_{j+1},\max(u_{j+1},v_{j+2,k})))\leq \\
C\Jp(u_j,u_{j+1})+C\Jp(u_{j+1},v_{j+2,k})\leq \\
C\Jp(u_j,u_{j+1})+C^2\Jp(u_{j+1},u_{j+2})+C^2\Jp(u_{j+2},v_{j+2,k})\leq \dots \leq \\
\sum_{l=0}^{k-j-1}C^{l+1}\Jp(u_{j+l},u_{j+l+1})\leq \sum_{l=0}^{k-j-1}C^{l+1}\frac {1}{(2C)^{j+l+1}}\leq \frac {1}{2^j}.
\end{multline}
From (\ref{4}), and (\ref{5}), it now follows
\begin{multline}\label{6}
\Jp(0,v_{j,k})\leq C(\Jp(0,u_j)+\Jp(u_j,u_{v_{j,k}}))\leq \\ C\left(C\Jp(0,u_1)+1+\frac {1}{2^j}\right)\leq
C^2(\Jp(u_1,0)+2).
\end{multline}
The sequence $v_{j,k}$ is increasing in $k$, and therefore it follows from (\ref{6}) that \newline $\sup_{k}e_{p,m}(v_{j,k})<\infty$. Hence, $v_j=\left(\lim_{k\to \infty}v_{j,k}\right)^*\in \mathcal E_{p,m}(\Omega)$. Furthermore, $v_j$ is a decreasing sequence, $v_j\searrow u=(\limsup_{j\to \infty}u_j)^*$. Again using (\ref{6}) we can conclude $\sup_{j}e_{p,m}(v_j)<\infty$. Thus, $u\in \mathcal E_{p,m}(\Omega)$.
We do not know if the quasimetric $\Jp$ is continuous under increasing sequences, but instead we can obtain some estimates. The monotone convergence theorem, and Proposition~\ref{weakconvergence} implies
\begin{multline}\label{7}
\liminf_{k\to \infty}\Jp(u_j,v_{j,k})^{p+m}=\liminf_{k\to \infty}\int_{\Omega}(v_{j,k}-u_j)^p\HH(v_{j,k})+ \\ \lim_{k\to \infty}\int_{\Omega}(v_{j,k}-u_j)^p\HH(u_j)\geq
\int_{\Omega}(v_j-u_j)^p\HH(v_j)+\int_{\Omega}(v_j-u_j)^p\HH(u_j)= \\ \Jp(u_j,v_j)^{p+m}.
\end{multline}
Hence, $\Jp(u_j,v_j)\to 0$, as $j\to \infty$, by using (\ref{5}) and (\ref{7}). Finally, $\varphi_j$ tends to $u$ in the quasimetric $\Jp$, since Proposition~\ref{monotone}, and (\ref{8}) yields
\[
\Jp(\varphi_j, u)\leq C^2(\Jp(u_j,v_j)+\Jp(v_j,u)+\Jp(u_j,\varphi_j))\to 0,
\]
as $j\to \infty$.
\end{proof}

\section{Convergence in the space $(\mathcal E_{p,m}(\Omega),\Jp)$}

 In this section we shall continue to study the convergence in $(\mathcal E_{p,m}(\Omega),\Jp)$. From Proposition~\ref{monotone} we know that that quasimetric $\Jp$ is continuous under decreasing sequences, and if $p=1$, then we also know continuity under increasing sequences. A summary of this section is as follows:

\smallskip

\begin{enumerate}\itemsep2mm

\item If $\Jp(u_j,u)\to 0$, as $j\to \infty$, then $u_j\to u$ in $L^{p+m}(\Omega)$ (Proposition~\ref{weak}).

\item If $\Jp(u_j,u)\to 0$, as $j\to \infty$,  then $u_j\to u$ in capacity $\operatorname {cap}_m$ (Proposition~\ref{cap}).

\item The inverse implications of the above results are not, in general valid (Example~\ref{conv_example}).

\item If $\Jp(u_j,u)\to 0$, as $j\to \infty$, then $\HH(u_j)\to \HH(u)$ weakly (Proposition~\ref{prop_weak}).
\end{enumerate}

\begin{proposition}\label{weak} Let $n\geq 2$, $1\leq m\leq n$, and assume that $\Omega\subset \mathbb C^n$ is a $m$-hyperconvex domain and $u_j,u\in \mathcal E_{p,m}(\Omega)$. If $\Jp(u_j,u)\to 0$, $j\to \infty$, then $u_j\to u$, $j\to \infty$, in  $L^{p+m}(\Omega)$.
\end{proposition}
\begin{proof} Recall that $\mathcal E_{p,m}(\Omega)\subset L^{p+m}(\Omega)$ (see e.g.~\cite{AC2020}). Let $\varphi\in \mathcal E_{0,m}(\Omega)$ be such that $\HH(\varphi)=dV_{2n}$. Assume that $\Jp(u_j,u)\to 0$, as $j\to \infty$. By Proposition~\ref{basic} we get
\[
\Jp(\max (u_j,u),u_j)\leq \Jp(u_j,u), \ \ \text {and} \ \ \Jp(\max(u_j,u),u)\leq \Jp(u_j,u),
\]
which implies that
\begin{equation}\label{10}
\Jp(\max (u_j,u),u_j)\to 0, \ \ \Jp(\max (u_j,u),u)\to 0, \ \ j\to \infty.
\end{equation}
Thanks to B\l ocki's inequality (see e.g.~\cite{L3, thien2, WanWang2}), together with (\ref{10}), we obtain
\begin{multline*}
\int_{\Omega}|u_j-u|^{p+m}dV_{2n}=\int_{\Omega}|u_j-u|^{p+m}\HH(\varphi)=\\
\int_{\{u_j<u\}}(u-u_j)^{p+m}\HH(\varphi)+\int_{\{u<u_j\}}(u_j-u)^{p+m}\HH(\varphi)=\\
\int_{\{u_j<u\}}(\max(u,u_j)-u_j)^{p+m}\HH(\varphi)+\int_{\{u<u_j\}}(\max(u_j,u)-u)^{p+m}\HH(\varphi)\leq \\
\int_{\Omega}(\max(u,u_j)-u_j)^{p+m}\HH(\varphi)+\int_{\Omega}(\max(u_j,u)-u)^{p+m}\HH(\varphi)\leq \\
(p+m)\dots(p+1)\|\varphi\|_{\infty}^m\Bigg(\int_{\Omega}(\max(u,u_j)-u_j)^{p}\HH(u_j)+\\
\int_{\Omega}(\max(u_j,u)-u)^{p}\HH(u)\Bigg)\leq \\
(p+m)\dots(p+1)\|\varphi\|_{\infty}^m\left(\Jp(\max (u_j,u),u_j)^{p+m}+\Jp(\max (u_j,u),u)^{p+m}\right)\to 0,
\end{multline*}
as $j\to\infty$.
\end{proof}

 With $u_j\to u$ in capacity $\operatorname{cap}_m$  we mean  that for any $K\Subset \Omega$, and any $\epsilon>0$, it holds
\[
\lim_{j\to \infty}\operatorname{cap}_m\left(K\cap\{z\in \Omega: |u_j(z)-u(z)|>\epsilon\}\right)=0,
\]
where the capacity of a Borel set $B\Subset \Omega$ is defined by
\[
\operatorname{cap}_m(B)=\sup\left\{\int_{B}\HH(\varphi): \varphi\in \mathcal {SH}_m(\Omega); \ -1\leq \varphi\leq 0\right\}.
\]

\begin{proposition}\label{cap} Let $n\geq 2$, $1\leq m\leq n$, and assume that $\Omega\subset \mathbb C^n$ is a $m$-hyperconvex domain, and  $u_j,u\in \mathcal E_{p,m}(\Omega)$. If $\Jp(u_j,u)\to 0$, as $j\to\infty$, then $u_j\to u$ in capacity $\operatorname {cap}_m$.
\end{proposition}
\begin{proof} Set $\varphi_j=(\sup_{k\geq j}u_j)^*$. Then it follows
$\varphi_j\searrow v$, and $\varphi_j\geq u_j$. Since $\varphi_j$ is decreasing sequence, we get from~\cite{L3} that $\varphi_j\to v$ in capacity $\operatorname{cap}_m$. Theorem~\ref{thm_holder} yields
\begin{multline*}
e_{p,m}(\varphi_j)^{\frac 1{p+m}}\leq D(p,m)^{\frac {1}{p}}e_{p,m}(u_j)^{\frac {1}{p+m}}=D(p,m)^{\frac {1}{p}}\Jp(u_j,0)\leq \\
D(p,m)^{\frac {1}{p}}C(\Jp(u_j,u)+\Jp(u,0)),
\end{multline*}
which means that
\[
\sup_{j}e_{p,m}(\varphi_j)^{\frac {1}{p+m}}\leq D(p,m)^{\frac {1}{p}}\sup_{j}e_{p,m}(u_j)^{\frac {1}{p+m}}<\infty.
\]
Hence, $v\in \mathcal E_{p,m}(\Omega)$. Moreover by Proposition~\ref{weak} we know that $u_j\to u$ in $L^{p+m}(\Omega)$, $j\to \infty$, and therefore also $\varphi_j\to u$ in $L^{p+m}(\Omega)$. This implies that $u=v$. We have by Proposition~\ref{monotone}
\begin{equation}\label{12}
\Jp(u_j,\varphi_j)\leq C(\Jp(u_j,u)+\Jp(\varphi_j,u))\to 0,
\end{equation}
Now observe that
\begin{multline*}
\{z\in \Omega: |u_j(z)-u(z)|>\epsilon\}\subset \\ \left\{z\in \Omega: |\varphi_j(z)-u_j(z)|>\frac {\epsilon}{2}\right\}\cup \left\{z\in \Omega: |\varphi_j(z)-u(z)|>\frac {\epsilon}{2}\right\}.
\end{multline*}
Therefore it is sufficient to prove
\[
\lim_{j\to \infty}\operatorname{cap}_m\left(K\cap\left\{z\in \Omega: |\varphi_j(z)-u_j(z)|>\frac {\epsilon}{2}\right\}\right)=0.
\]
Let  $\psi\in \mathcal E_{0,m}(\Omega)$ be such that $-1\leq \psi\leq 0$, and $K\Subset \Omega$. Then by B\l ocki's inequality (see e.g. ~\cite{L3, thien2, WanWang2}) and by (\ref{12}) we get
\begin{multline*}
\int_{K\cap\{z\in \Omega: |\varphi_j(z)-u_j(z)|>\frac {\epsilon}{2}\}}\HH(\psi)\leq \\ \frac{2^{m+p}}{\epsilon ^{m+p}}\int_{K\cap\{z\in \Omega: |\varphi_j(z)-u_j(z)|>\frac {\epsilon}{2}\}}(\varphi_j-u_j)^{m+p}\HH(\psi)\leq\\
\frac{2^{m+p}}{\epsilon ^{m+p}}\int_{\Omega}(\varphi_j-u_j)^{m+p}\HH(\psi)\leq \\ \frac{2^{m+p}}{\epsilon ^{m+p}}(p+m)\dots(p+1)\|\psi\|^m_{\infty}\int_{\Omega}(\varphi_j-u_j)^p\HH(u_j)\leq \\
\frac{2^{m+p}(p+m)\dots(p+1)\|\psi\|^m_{\infty}}{\epsilon ^{m+p}}\Jp(u_j,\varphi_j)^{p+m}\to 0,
\end{multline*}
as $j\to \infty$.
\end{proof}

Note that the reverse implications in Proposition~\ref{weak} and Proposition~\ref{cap} are not, in general true. The following example is taken from~\cite{CAV}.

\begin{example}\label{conv_example} Let
\[
u_j(z)=\max\left(j^{\frac pn}\ln |z|,-\frac 1j\right)
\]
be a plurisubharmonic function defined in the unit ball $\mathbb B$ in $\mathbb C^n$, $n\geq 2$. Then $u_j\in \mathcal E_{0,n}(\mathbb B)$,
\[
\Jp(u_j,0)^{p+n}=e_{p,n}(u_j)=(2\pi)^n,
\]
but $u_j\to 0$ in capacity, and $u_j\to 0$ in $L^{p+n}(\mathbb B)$.\hfill{$\Box$}
\end{example}

At the end of this section, we shall prove that convergence in $\Jp$ implies weak convergence of the complex Hessian measures. We shall need the following well-known lemma.

\begin{lemma}\label{lem1} Let $n\geq 2$, $1\leq m\leq n$, and assume that $\Omega\subset \mathbb C^n$ is a $m$-hyperconvex domain and  $u_j,u\in \mathcal E_{p,m}(\Omega)$ be such that $\sup_je_{p,m}(u_j)<\infty$. Then, if $u_j\to u$ in capacity $\operatorname {cap}_m$, then $\HH(u_j)\to \HH(u)$ weakly, as $j\to \infty$.
\end{lemma}

\begin{proof} Set
\[
A=e_{p,m}(u)+\sup_je_{p,m}(u_j)<\infty.
\]
Then for all $k,j$, we have
\[
e_{p,m}(\max(u_j,-k))\leq D(p,m)^{\frac {m+p}{p}} e_{p,m}(u_j)\leq D(m,p)^{\frac {m+p}{p}} A.
\]
Consider the following decomposition
\begin{multline*}
\HH(u_j)-\HH(u)=\left(\HH(u_j)-\HH(\max(u_j,-k))\right)+\\
\left(\HH(\max(u_j,-k))-\HH(\max(u,-k))\right)+\left(\HH(\max(u,-k))-\HH(u)\right)=\\
=\mu_{j,k}^1+\mu_{j,k}^2+\mu_{k}^3.
\end{multline*}
Furthermore,  since $u_j\to u$ in capacity $\operatorname {cap}_m$, then for all $k$ we get $\max(u_j,-k)\to \max(u,-k)$, in capacity $\operatorname {cap}_m$. All function are uniformly bounded, so by~\cite{L3}, $\mu_{j,k}^2\to 0$ weakly as $j\to \infty$. Since, $\max(u,-k)$ is decreasing to $u$, as $k\to \infty$, and therefore $\mu_k^3\to 0$ weakly, as $k\to \infty$.

To finish the proof we have to show that $\mu_{j,k}^1\to 0$, as $k\to \infty$ and uniformly on $j$. Let $\alpha\in \mathcal C_0^{\infty}(\Omega)$, and we
shall use the temporary notation $T_s=(dd^c\max(u_j,-k))^{m-s-1}\wedge\be$. Then
\begin{multline*}
\left|\int_{\Omega}\alpha \mu^1_{j,k}\right|=
\left|\sum_{s=0}^{m-1}\int_{\Omega}\alpha (dd^cu_j)^s\wedge(dd^cu_j-dd^c\max(u_j,-k))\wedge T_s\right|\leq\\
\|\alpha\|_{\infty}k^{-p}\sum_{s=0}^{m-1}\int_{\{u_j\leq -k\}}(-u_j)^p(dd^cu_j)^{s+1}\wedge T_s +\\
\|\alpha\|_{\infty}k^{-p}\sum_{s=0}^{m-1}\int_{\{u_j\leq -k\}}(-u_j)^p(dd^cu_j)^s\wedge (dd^c\max(u_j,-k))^{m-s}\wedge\be \leq \\
\|\alpha\|_{\infty}k^{-p}\sum_{s=0}^{m-1}D(m,p)e_{p,m}(u_j)^{\frac {p+s+1}{p+m}}e_{p,m}(\max(u_j,-k))^{\frac {m-s-1}{p+m}}+ \\
\|\alpha\|_{\infty}k^{-p}\sum_{s=0}^{m-1}D(m,p)e_{p,m}(u_j)^{\frac {p+s}{p+m}}e_{p,m}(\max(u_j,-k))^{\frac {m-s}{p+m}}\leq \\
\|\alpha\|_{\infty}k^{-p}2mD(m,p)^{\frac {m+p}{p}}A\to 0,
\end{multline*}
as $k\to \infty$. The convergence above is uniform in $j$.
\end{proof}

\begin{proposition}\label{prop_weak} Let $n\geq 2$, $1\leq m\leq n$, and assume that $\Omega\subset \mathbb C^n$ is a $m$-hyperconvex domain. If $\Jp(u_j,u)\to 0$, as $j\to \infty$, then $\HH(u_j)\to \HH(u)$ weakly.
\end{proposition}
\begin{proof} Since
\[
e_{p,m}(u_j)^{\frac 1{p+m}}=\Jp(u_j,0)\leq C(\Jp(u_j,u)+\Jp(u,0)),
\]
it follows
\[
e_{p,m}(u)+\sup_{j}e_{p,m}(u_j)<\infty.
\]
This proof is then concluded by Proposition~\ref{cap}, and Lemma~\ref{lem1}.
\end{proof}

\section{A comparison of different topologies}

In this section, we begin by comparing the quasimetric space $(\mathcal E_{1,m}(\Omega), \operatorname {J}_1)$ with the metric space $(\mathcal E_{1,m}(\Omega),\dd)$ studied in~\cite{ACmetric}. In the second part of this section, we show that the topology generated by $\Jp$ is not comparable with the topology generated by the subspace $\mathcal{E}_{p,m}(\Omega)$ of quasi-normed space $(\delta\mathcal{E}_{p,m}(\Omega), \|\cdot\|_p)$ studied in~\cite{thien}. The presentation of the latter part follows closely Section~7 of~\cite{ACmetric}.

Let us first introduce the necessary definitions and notations and formulate the results that are relevant here. For further information see~\cite{ACmetric}.

\begin{definition}\label{defef} Let $1\leq m\leq n$, and let $\Omega$ be a bounded $m$-hyperconvex domain in $\mathbb C^n$, $n>1$. Fix $w\in \mathcal E_{1,m}(\Omega)$, known as the weight, and define the weighted energy functional $\EE$ by
\begin{equation}\label{energy}
\mathcal E_{1,m}(\Omega)\ni u\mapsto \EE(u)=\frac {1}{m+1}\sum_{j=0}^m\int_{\Omega}(u-w)(dd^cu)^j\wedge(dd^cw)^{m-j}\wedge\be \in \mathbb R.
\end{equation}
\end{definition}

\begin{definition}\label{metric} Let $n\geq 2$, $1\leq m\leq n$, and assume that $\Omega\subset \mathbb C^n$ is a $m$-hyperconvex domain. Fix $w,u,v\in \mathcal E_{1,m}(\Omega)$. Let us define
\[
\dd (u,v)=\EE(u)+\EE(v)-2\EE(\PP(u,v)).
\]
\end{definition}

\begin{theorem}\label{thm_compmetric}
Let $n\geq 2$, $1\leq m\leq n$, and let $\Omega$ be a bounded $m$-hyperconvex domain in $\mathbb C^n$. Then the tuple $(\mathcal E_{1,m}(\Omega),\dd)$ is a complete metric space.
\end{theorem}

Now we are in position to compare the metric $\dd$ with the quasimetric $\operatorname{J}_1$.

\begin{theorem} Let $n\geq 2$, $1\leq m\leq n$, and assume that $\Omega\subset \mathbb C^n$ is a $m$-hyperconvex domain. Then for $u,v\in \mathcal E_{1,m}(\Omega)$ it holds
\[
\frac {1}{2^{m+2}C^{m+1}(m+1)}\operatorname{J}_1(u,v)^{m+1}\leq \dd(u,v)\leq \operatorname{J}_1(u,v)^{m+1},
\]
where $C$ is a constant from quasi-triangle inequality.
\end{theorem}
\begin{proof}
Let $u,v\in \mathcal E_{1,m}(\Omega)$. Note that $\{\PP(u,v)=u\}\subset \{u\leq v\}$ and $\{\PP(u,v)=v\}\subset \{v\leq u\}$. Then we obtain by Theorem~\ref{mp} and~\cite[Proposition 3.3]{ACmetric}
\begin{multline*}
\dd(u,v)=\EE(u)-\EE(\PP(u,v))+\EE(v)-\EE(\PP(u,v))\leq \\
\int_{\Omega}(u-\PP(u,v))\HH(\PP(u,v))+\int_{\Omega}(v-\PP(u,v))\HH(\PP(u,v))=\\
\int_{\{\PP(u,v)<u\}}(u-\PP(u,v))(\chi_{\{\PP(u,v)=v\}}\HH(v)+\chi_{\{\PP(u,v)=u\}}\HH(u))+\\
\int_{\{\PP(u,v)<v\}}(v-\PP(u,v))(\chi_{\{\PP(u,v)=v\}}\HH(v)+\chi_{\{\PP(u,v)=u\}}\HH(u))\leq\\
\int_{\{\PP(u,v)<u\}\cap \{\PP(u,v)=v\}}(u-\PP(u,v))\HH(v)+\\
\int_{\{\PP(u,v)<v\}\cap \{\PP(u,v)=u\}}(v-\PP(u,v))\HH(u)\leq\\
\int_{\{v<u\}}(u-v)\HH(v)+\int_{\{u<v\}}(v-u)\HH(u)\leq \\
\int_{\{v<u\}}(u-v)(\HH(v)+\HH(u))+\int_{\{u<v\}}(v-u)(\HH(u)+\HH(v))=\operatorname{J}_1(u,v)^{m+1}.
\end{multline*}
Thanks to Proposition~\ref{basic}, the quasi-triangle inequality, and~\cite[Proposition 3.3]{ACmetric} we get
\begin{multline*}
\dd(u,v)=\EE(u)-\EE(\PP(u,v))+\EE(v)-\EE(\PP(u,v))\geq \\
\frac {1}{m+1}\int_{\Omega}(u-\PP(u,v))\HH(\PP(u,v))+\frac {1}{m+1}\int_{\Omega}(v-\PP(u,v))\HH(\PP(u,v))\geq\\
\frac {1}{2(m+1)}\operatorname{J}_1(u,\PP(u,v))^{m+1}+\frac {1}{2(m+1)}\operatorname{J}_1(v,\PP(u,v))^{m+1}\geq \\
\frac {1}{2^{m+2}(m+1)}\left(\operatorname{J}_1(u,\PP(u,v))+\operatorname{J}_1(v,\PP(u,v))\right)^{m+1}\geq \frac {1}{2^{m+2}C^{m+1}(m+1)}\operatorname{J}_1(u,v)^{m+1}.
\end{multline*}
\end{proof}

Now to the second part of this section. Let us start here with a brief introduction. We start by defining
\[
\delta\mathcal E_{p,m}(\Omega)=\mathcal E_{p,m}(\Omega)-\mathcal E_{p,m}(\Omega),
\]
since $\mathcal E_{p,m}(\Omega)$ is only a convex cone. Then for any $u\in \delta\mathcal E_{p,m}(\Omega)$ we define
\[
\|u\|_p=\inf_{u_1-u_2=u \atop u_1,u_2\in \mathcal{E}_{p,m}(\Omega)}\left(\int_{\Omega} (-(u_1+u_2))^p\operatorname{H}_m(u_1+u_2)
\right)^{\frac {1}{m+p}}\, .
\]
It was proved in~\cite{thien} that $(\delta\mathcal{E}_{p,m}(\Omega), \|\cdot\|_p)$ is a quasi-Banach space, i.e. it is complete quasi-normed vector space (for the case $m=n$ see~\cite{mod}). Recall that $\|\cdot\|_p$ is a quasinorm if the following holds:

\medskip

\noindent\begin{enumerate}\itemsep2mm
\item $\|u\|_p=0$ if, and only if, $u=0$;

\item $\|tu\|_p=|t| \|u\|_p$;

\item it satisfies quasi-triangle inequality
\[
\|u+v\|_p\leq C(\|u\|_p+\|v\|_p),
\]
for some constant $C\geq 1$.
\end{enumerate}
\noindent Furthermore, if $u\in \mathcal E_{p,m}(\Omega)$, then $\|u\|_p=e_{p,m}(u)^{\frac {1}{p+m}}$.

\begin{example}
There is no constant $C>0$ such that $\Jp(u,v)\leq C\|u-v\|_p$. To see this let $w\in \mathcal E_{p,m}(\Omega)$, and $t>0$. Then
\begin{multline*}
\Jp(w+tw,tw)^{p+m}=\int_{\Omega}|w+tw-tw|^p(\HH((t+1)w)+\HH(tw))=\\
((t+1)^m+t^m)e_{p,m}(w)\to \infty,
\end{multline*}
as $t\to \infty$, but
\[
\|w+tw-tw\|_p=\|w\|_p=e_{p,m}(w)^{\frac {1}{p+m}}<\infty.
\]
\hfill{$\Box$}
\end{example}

\begin{example}
There is no constant $C>0$ such that $\|u-v\|_p\leq C\Jp(u,v)$. To see this define the following  plurisubharmonic functions in unit ball $\mathbb B$ in $\mathbb C^n$:
\[
u_j(z)=j\max (\log |z|, a_j), \ \ v_j(z)=j\max (\log |z|, c),
\]
where $a_j\searrow c$, as $j\to \infty$, where $a_j,c<0$ and the sequence $a_j=c+j^{-\frac {n+p+1}{p}}$. By~\cite{mod} we have
\begin{multline*}
\|u_j-v_j\|^{n+p}=\|u_j+v_j\|^{n+p}=e_{p,m}(u_j+v_j)\geq \int_{\Omega}(-u_j-v_j)^p(dd^cv_j)^n=\\
(2\pi)^nj^{n+p}(-a_j-c)^p\to \infty,
\end{multline*}
as $j\to \infty$. Since, $u_j\geq v_j$ it follows
\[
\Jp(u_j,v_j)^{p+n}=\int_{\mathbb B}|u_j-v_j|^p((dd^cu_j)^n+(dd^cv_j)^{n})=
(2\pi)^nj^{n+p}(a_j-c)^p,
\]
then it holds that $\Jp(u_j,v_j)\to 0$, as $j\to \infty$.\hfill{$\Box$}
\end{example}

\section{Stability of the complex Hessian operator}\label{sec_stability}

Let
\[
\begin{aligned}
\MM=\big\{&\mu \; :\;  \mu \text{ is a non-negative Radon measure on $\Omega$ such that }\\
 & \HH(u)=\mu \text{ for some } u\in \mathcal E_{p,m}\big \},
\end{aligned}
\]
and recall the following characterization of $\MM$:

\medskip

\begin{enumerate}\itemsep2mm
\item $\mu\in \MM$;
\item there exists a constant $A\geq 0$ such that
\[
\int_{\Omega}(-u)^p\,d\mu\leq A\, e_{p,m}(u)^{\frac{p}{p+m}}\quad \text{ for all } u\in\mathcal E_{p,m}(\Omega)\, ;
\]
\item $\mathcal E_{p,m}(\Omega)\subset L^p(\mu)$;

\item there exists unique function $U(\mu)\in \mathcal E_{p,m}(\Omega)$ the solution to the Dirichlet problem for the complex Hessian equation $\HH(U(\mu))=\mu$,

\end{enumerate}
(see~\cite{cegrell_pc,L3} for details).

In this section we shall prove some new stability results for the complex Hessian operator. In previous results concerning stability of the complex Monge-Amp\`ere equation or the complex Hessian equation see e.g.~\cite{CK,C,thien2} In those papers the authors proved that under some assumption if $\mu_j$ converges to $\mu$, then the corresponding solutions $U(\mu_j)$ converges to $U(\mu)$ in capacity. Our goal is to prove that the convergence is stronger in the sense that $\Jp(U(\mu_j),U(\mu))\to 0$.

\begin{lemma}\label{lem2} Let $n\geq 2$, $1\leq m\leq n$, and assume that $\Omega\subset \mathbb C^n$ is a $m$-hyperconvex domain. Let $\mu\in \MM$. Then for any $\nu\leq \mu$ it holds:
\[
e_{p,m}(U(\nu))\leq D(p,m)^{\frac {m+p}{p}}e_{p,m}(U(\mu)).
\]
\end{lemma}
\begin{proof}
The desired inequality follows from the following estimation
\begin{multline*}
e_{p,m}(U(\nu))=\int_{\Omega}(-U(\nu))^p\HH(U(\nu))=\int_{\Omega}(-U(\nu))^pd\nu\leq \int_{\Omega}(-U(\nu))^pd\mu=\\
\int_{\Omega}(-U(\nu))^p\HH(U(\mu))\leq D(m,p)e_{p,m}(U(\nu))^{\frac {p}{p+m}}e_{p,m}(U(\mu))^{\frac {m}{m+p}}.
\end{multline*}
\end{proof}

\begin{definition}\label{fund seq}
A fundamental sequence $\Omega_{j}$, $j\in \mathbb N$, is an increasing sequence of $m$-hyperconvex subsets of
$\Omega\subset\C^n$, $n\geq 2$, such that for every $j\in\N$ we have that, $\Omega_{j}\Subset\Omega_{j+1}$, and
$\bigcup_{j=1}^{\infty}\Omega_{j}=\Omega$.
\end{definition}

The main result in this section is the following stability theorem.

\begin{theorem}\label{thm_stability} Let $n\geq 2$, $1\leq m\leq n$, and assume that $\Omega\subset \mathbb C^n$ is a $m$-hyperconvex domain and let $\mu\in \MM$. If
$0\leq f, f_j\leq 1$ are measurable functions such that $f_j\to f$ in $L^1_{loc}(\mu)$, as $j\to +\infty$, then $\Jp(U(f_j\mu), U(f\mu))\to 0$.
\end{theorem}
\begin{proof}
Fix $\mu \in \MM$.  From the Cegrell-Lebesgue decomposition theorem (see~\cite{cegrell_pc, L3}) it follows that there exist $\varphi\in \mathcal E_{0,m}(\Omega)$, $\|\varphi\|_{\infty}\leq 1$ and $g\geq 0$ such that $g\HH(\varphi)=\mu$. Fix a fundamental sequence $\Omega_j$. For $j,k\in \mathbb N$, let us define the following functions:

\begin{alignat*}{2}
&w=U(\mu)\in \mathcal E_{p,m}(\Omega): \ \ &&\HH(w)=\mu; \\
&u=U(f\mu)\in \mathcal E_{p,m}(\Omega): \ \ &&\HH(u)=f\mu;\\
&u_j=U(f_j\mu)\in \mathcal E_{p,m}(\Omega): \ \ &&\HH(u_j)=f_j\mu; \\
&u_{j,k}\in \mathcal E_{p,m}(\Omega): \ \ &&\HH(u_{j,k})=\chi_{\Omega_k}f_j\min(g,k)\HH(\varphi); \\
&v_{j,k}\in \mathcal E_{p,m}(\Omega): \ \ &&\HH(v_{j,k})=\chi_{\Omega_k}f_j\mu; \\
&w_{k}\in \mathcal E_{p,m}(\Omega): \ \ &&\HH(w_{k})=\chi_{\Omega_k}f\min(g,k)\HH(\varphi); \\
&v_{k}\in \mathcal E_{p,m}(\Omega): \ \ &&\HH(v_{k})=(1-\chi_{\Omega_k})\mu; \\
&\psi_{k}\in \mathcal E_{p,m}(\Omega): \ \ &&\HH(\psi_{k})=(g-\min(g,k))\HH(\varphi). \\
\end{alignat*}

Since,
\begin{equation}\label{11}
\Jp(u,u_j)\leq C^3(\Jp(u,w_k)+\Jp(w_k,u_{j,k})+\Jp(u_{j,k},v_{j,k})+\Jp(v_{j,k},u_j)),
\end{equation}
it is enough to prove that each term in (\ref{11}) tend to zero to complete the proof.

\emph{Claim 1}. $\Jp(u,w_k)\to 0$, as $k\to \infty$. This follows from that $w_k$ is a decreasing sequence tending to $u$, as $k\to \infty$.

\emph{Claim 2}. For fixed $k$ we have that $\Jp(w_k,u_{j,k})\to 0$, as $j\to \infty$. To prove this first note that the functions $w_k$ and $u_{j,k}$ are uniformly bounded by $k^{\frac 1m}$, and therefore it follows
\begin{multline*}
\Jp(w_k,u_{j,k})^{p+m}=\int_{\Omega}|w_k-u_{j,k}|^p(\HH(w_k)+\HH(u_{j,k}))=\\
\int_{\{w_k<u_{j,k}\}}(u_{j,k}-w_k)^p\chi_{\Omega_k}(f_j+f)\min(g,k)\HH(\varphi)+\\
\int_{\{w_k>u_{j,k}\}}(w_k-u_{j,k})^p\chi_{\Omega_k}(f_j+f)\min(g,k)\HH(\varphi)\leq \\
2k\int_{\{w_k<u_{j,k}\}}(u_{j,k}-w_k)^p\HH(\varphi)+2k\int_{\{w_k>u_{j,k}\}}(w_k-u_{j,k})^p\HH(\varphi).
\end{multline*}

Now assume that $p\leq m$. Then we can continue our estimate by using the H\"older inequality. By~\cite{thien2} we get
\begin{multline*}
\int_{\{w_k<u_{j,k}\}}(u_{j,k}-w_k)^p\HH(\varphi)\leq\\ \left(\int_{\{w_k<u_{j,k}\}}(u_{j,k}-w_k)^m\HH(\varphi)\right)^{\frac {p}{m}}(\HH(\varphi)(\Omega))^{\frac {m-p}{m}}\leq\\
(\HH(\varphi)(\Omega))^{\frac {m-p}{m}}\left(m!\int_{\{w_k<u_{j,k}\}}(-\varphi)(\HH(w_k)-\HH(u_{j,k}))\right)^{\frac {p}{m}}\leq \\
(\HH(\varphi)(\Omega))^{\frac {m-p}{m}}\left(m!\int_{\Omega_k}|f_j-f|d\mu\right)^{\frac {p}{m}}\to 0, \text{ as } j\to \infty.
\end{multline*}

In a similar way, one can prove
\[
\int_{\{w_k>u_{j,k}\}}(w_k-u_{j,k})^p\HH(\varphi)\leq (\HH(\varphi)(\Omega))^{\frac {m-p}{m}}\left(m!\int_{\Omega_k}|f_j-f|d\mu\right)^{\frac {p}{m}}\to 0,
\]
as $j\to \infty$. If $p>m$, then one can repeat the above argument using the fact
\[
|u_{j,k}-w_k|^p\leq (2k^{\frac 1m})^{p-m}|u_{j,k}-w_k|^m.
\]

\emph{Claim 3}. $\Jp(v_{j,k},u_{j,k})\to 0$, as $k\to \infty$, and the convergence is uniform on $j$. Since
\begin{multline*}
\HH(v_{j,k})=\chi_{\Omega_k}f_jg\HH(\varphi)\leq \\
\chi_{\Omega_k}f_j\min(g,k)\HH(\varphi)+(g-\min(g,k))\HH(\varphi)=\\
\HH(u_{j,k})+\HH(\psi_k)\leq \HH(u_{j,k}+\psi_k),
\end{multline*}
we have that $u_{j,k}+\psi_k\leq v_{j,k}$. Furthermore, $u_{j,k}\geq v_{j,k}$, and $\psi_k$ is a increasing sequence such that
\begin{multline*}
e_{p,m}(\psi_k)=\int_{\Omega}(-\psi_k)^p\HH(\psi_k)=\int_{\Omega}(-\psi_k)^p(g-\min(g,k))\HH(\varphi)\leq \\
\int_{\Omega}(-\psi_1)^p(g-\min(g,k))\HH(\varphi)\to 0, \ k\to \infty,
\end{multline*}
by dominated convergence theorem. Finally,
\begin{multline*}
\Jp(v_{j,k},u_{j,k})^{p+m}=\int_{\Omega}|v_{j,k}-u_{j,k}|^p(\HH(v_{j,k})+\HH(u_{j,k}))\leq \\
2\int_{\Omega}(u_{j,k}-v_{j,k})^p\HH(w)\leq 2\int_{\Omega}(-\psi_k)^p\HH(w)\leq \\
2D(m,p)e_{p,m}(\psi_k)^{\frac {p}{m+p}}e_{p,m}(w)^{\frac {m}{m+p}}\to 0,
\end{multline*}
as $k\to \infty$. The convergence is as well uniform in $j$.

\emph{Claim 4}. $\Jp(v_{j,k},u_{j})\to 0$, as $k\to \infty$, and the convergence is uniform on $j$. The proof of Claim 4 follows that of Claim 3.
Observe that
\begin{multline*}
\HH(u_{j})=f_j\mu\leq (1-\chi_{\Omega_k})\mu+f_j\chi_{\Omega_k}\mu=\\
\HH(v_k)+\HH(v_{j,k})\leq \HH(v_k+v_{j,k}),
\end{multline*}
which implies that $v_k+v_{j,k}\leq u_{j}$. Furthermore, $u_{j}\leq v_{j,k}$, and $v_k$ is increasing sequence such that
\begin{multline*}
e_{p,m}(v_k)=\int_{\Omega}(-v_k)^p\HH(v_k)=\int_{\Omega}(-v_k)^p(1-\chi_{\Omega_k})\mu\leq \\
\int_{\Omega}(-v_1)^p(1-\chi_{\Omega_k})\mu\to 0, \ k\to \infty,
\end{multline*}
by dominated convergence theorem. Finally,
\begin{multline*}
\Jp(v_{j,k},u_{j})^{p+m}=\int_{\Omega}|v_{j,k}-u_{j}|^p(\HH(v_{j,k})+\HH(u_{j}))\leq \\
2\int_{\Omega}(v_{j,k}-u_{j})^p\HH(w)\leq 2\int_{\Omega}(-v_k)^p\HH(w)\leq \\
2D(m,p)e_{p,m}(v_k)^{\frac {p}{m+p}}e_{p,m}(w)^{\frac {m}{m+p}}\to 0, \ k\to \infty,
\end{multline*}
the convergence is uniform in $j$.

\end{proof}

\section{The compact K\"ahler manifold case}\label{sec_kah}

Let $n\geq 2$, $p>0$, and let $1\leq m\leq n$. Assume that $(X,\omega)$ is a connected and compact K\"ahler manifold of complex dimension $n$, where $\omega$ is a K\"ahler form on $X$ such that $\int_{X}\omega^n=1$. In a similar way as in Section~\ref{sec_quasimetric}, we define a quasimetric, $I_p$, for $(\omega,m)$-subharmonic functions (for the case $m=n$, see~\cite{GLZ}). For further information concerning $(\omega,m)$-subharmonic function ($\mathcal {SH}_m(X,\omega)$) on compact K\"ahler manifold see e.g.~\cite{ACchar,DL,GLZ,LN}.

\medskip

For any $u\in \mathcal {SH}_m(X,\omega)$, let
\[
\omega_u=dd^cu+\omega.
\]
The complex Hessian operator is defined on  $(\omega,m)$-subharmonic functions through the following construction:  First assume that $u\in \mathcal {SH}_m(X,\omega)\cap L^{\infty}(X)$, then
\[
\operatorname{H}_m(u):=\omega_u^m\wedge \omega^{n-m},
\]
which is a non-negative (regular) Borel measure defined on $X$. For an arbitrary, not necessarily bounded, $(\omega,m)$-subharmonic function $u$ let $u_j=\max(u,-j)$ be the canonical approximation of $u$. Then define
\[
\operatorname{H}_m(u):=\lim_{j\to \infty}\chi_{\{u>-j\}}\operatorname{H}_m(u_j).
\]
We define the class of \emph{$(\omega,m)$-subharmonic functions with bounded $(p,m)$-energy} as
\[
\mathcal E_{p,m}(X,\omega):=\left\{u\in \mathcal E_m(X,\omega): u\leq 0, \int_X(-u)^p\operatorname{H}_m(u)<\infty\right\},
\]
where
\[
\mathcal E_m(X,\omega)=\left\{u\in \mathcal {SH}_m(X,\omega): \int_X\operatorname{H}_m(u)=1 \right\}.
\]
For the counterparts of Theorem~\ref{thm_holder}, Theorem~\ref{thm_prel}, and the approximation theorem for $(\omega,m)$-subharmonic functions defined on compact K\"ahler manifolds we refer to \cite{ACchar,DL,GLZ,LN,LN2}.

The following definition is the counterpart of $\Jp$ in Definition~\ref{def_Jp}.

\begin{definition}
Let $n\geq 2$, $p>0$, and let $1\leq m\leq n$. For $u,v\in \mathcal E_{p,m}(X,\omega)$ and we define
\[
\Ip(u,v)=\left(\int_{X}|u-v|^p(\HH(u)+\HH(v))\right)^{\frac 1{p+m}}.
\]
\end{definition}
\begin{remark}
Note that it follows from~\cite[Lemma 3.5]{ACchar} that the functional $\Ip$ is well-defined.
\end{remark}

The aim of this section is to prove that $(\mathcal E_{p,m}(X,\omega),\Ip)$ is a complete quasimetric space.

\begin{theorem}\label{thmKah_quasimetric}
The pair $(\mathcal E_{p,m}(X,\omega),\Ip)$ is a quasimetric space.
\end{theorem}
\begin{proof}
First assume that $u,v\in \mathcal E_{p,m}(X,\omega)$, and $\Ip(u,v)=0$. Then $\operatorname{H}_m(u)(\{u<v\})=0$, and by~\cite[Theorem~3.15]{LN} we get $u\geq v$. In a similar way, we can obtain that $v\geq u$. Hence, $u=v$. The quasi-triangle inequality follows in the same way as in Lemma~\ref{qm}.
\end{proof}

We shall need the counterpart of Proposition~\ref{cp xing}.

\begin{proposition}\label{cp xing_m}
Let $u,v\in \mathcal E_{p,m}(X,\omega)$.
\begin{enumerate}\itemsep2mm
\item If $u\leq v$, then
\[
\int_{X}(v-u)^p\HH(v)\leq \int_{X}(v-u)^p\HH(u).
\]
\item Without any additional assumption on $u$, and $v$,  it holds
\[
\int_{\{u<v\}}(v-u)^p\HH(v)\leq \int_{\{u<v\}}(v-u)^p\HH(u).
\]
\end{enumerate}
\end{proposition}
\begin{proof} (1) First assume that $u<v$. Then for any positive current $T$ it holds
\[
\int_X(v-u)^p\omega_u\wedge T-\int_X(v-u)^p\omega_v\wedge T=p\int_X(v-u)^{p-1}d(v-u)\wedge d^c(v-u)\wedge T\geq 0.
\]
Now, let $\epsilon<1$, then $u\leq v<\epsilon v$. We obtain
\[
\epsilon \int_X(\epsilon v-u)^p\omega_v\wedge T\leq \int_X(\epsilon v-u)^p\omega_{\epsilon v}\wedge T\leq \int_X(\epsilon v-u)^p\omega_u\wedge T,
\]
so by the monotone convergence theorem, and letting $\epsilon \to 1^-$, we arrive ay
\[
\int_X(v-u)^p\omega_v\wedge T\leq \int_X(v-u)^p\omega_v\wedge T.
\]
Repeating the above argument $m$-times we get
\[
\int_{X}(v-u)^p\HH(v)\leq \int_{X}(v-u)^p\HH(u).
\]
(2) This part follows now in a similar manner as  in Proposition~\ref{cp xing}.
\end{proof}

\begin{corollary}\label{decr}
Let $u,u_j,v\in \mathcal E_{p,m}(X,\omega)$.
\begin{enumerate}
\item If $u\leq v$, then
\[
2\int_X(u-v)^p\HH(v)\leq \Ip(u,v)^{p+m}\leq 2\int_X(u-v)^p\HH(u).
\]
\item If $u_j\searrow u$, $j\to \infty$, then $\Ip(u_j,u)\to 0$.
\end{enumerate}
\end{corollary}

We end this paper with the main result of the compact K\"ahler case.

\begin{theorem}\label{thmKah_comp}
The pair $(\mathcal E_{p,m}(X,\omega),\Ip)$ is a complete quasimetric space.
\end{theorem}
\begin{proof}
Theorem~\ref{thmKah_quasimetric} gives that $(\mathcal E_{p,m}(X,\omega),\Ip)$ is a quasimetric space, and  the completeness can be proved in exactly the same way as in Theorem~\ref{comp}.
\end{proof}

\end{document}